\documentclass[onefignum,onetabnum]{siamonline220329}

\usepackage{lipsum}
\usepackage{amsfonts}
\usepackage{graphicx}
\usepackage{epstopdf}
\usepackage{algorithmic}
\usepackage{bm,nicefrac}
\ifpdf
  \DeclareGraphicsExtensions{.eps,.pdf,.png,.jpg}
\else
  \DeclareGraphicsExtensions{.eps}
\fi

\usepackage{enumitem}
\setlist[enumerate]{leftmargin=.5in}
\setlist[itemize]{leftmargin=.5in}

\newsiamremark{remark}{Remark}
\newsiamremark{example}{Example}
\newsiamremark{hypothesis}{Hypothesis}
\crefname{hypothesis}{Hypothesis}{Hypotheses}
\newsiamthm{claim}{Claim}

\headers{Constructing structured tensor priors for Bayesian inverse problems}{K. Batselier}

\title{
Constructing structured tensor priors for Bayesian inverse problems
\thanks{\funding{This publication is part of the project Sustainable learning for Artificial Intelligence from noisy large-scale data (with project number VI.Vidi.213.017) which is financed by the Dutch Research Council (NWO).}}}

\author{Kim Batselier\thanks{Delft Center for Systems and Control, Delft University of Technology, The Netherlands
  (\email{k.batselier@tudelft.nl})}
}

\usepackage{amsopn,booktabs}
\usepackage{caption}
\usepackage{subcaption}

\newcommand{\mat}[1]{\bm{#1}}
\newcommand{\ten}[1]{\bm{\mathcal{#1}}}

\newcommand{\vect}{\operatorname{vec}}

\ifpdf
\hypersetup{
  pdftitle={Constructing structured tensor priors for Bayesian inverse problems},
  pdfauthor={K. Batselier}
}
\fi

\begin{document}

\maketitle

\begin{abstract}
Specifying a prior distribution is an essential part of solving Bayesian inverse problems. The prior encodes a belief on the nature of the solution and this regularizes the problem. In this article we completely characterize a Gaussian prior that encodes the belief that the solution is a structured tensor. We first define the notion of $(\mat{A},\mat{b})$-constrained tensors and show that they describe a large variety of different structures such as Hankel, circulant, triangular, symmetric, and so on. Then we completely characterize the Gaussian probability distribution of such tensors by specifying its mean vector and covariance matrix. Furthermore, explicit expressions are proved for the covariance matrix of tensors whose entries are invariant under a permutation. These results unlock a whole new class of priors for Bayesian inverse problems. We illustrate how new kernel functions can be designed and efficiently computed and apply our results on two particular Bayesian inverse problems: completing a Hankel matrix from a few noisy measurements and learning an image classifier of handwritten digits. The effectiveness of the proposed priors is demonstrated for both problems. All applications have been implemented as reactive Pluto notebooks in Julia.
\end{abstract}



\section{Introduction}
We consider a set of data samples $\{(\mat{x}_n , y_n)\,|\, \mat{x}_n \in \mathbb{R}^{D}, \, y_n \in \mathbb{R} \}_{n=1}^N$ and the following linear forward model
\begin{align}
\label{eq:forwardtensormodel}
    {y}_n = \langle \ten{P}(\mat{x}_n), \ten{W} \rangle + \epsilon_n.
\end{align}
Each scalar measurement $y_n$ is obtained from an inner product of a data-dependent tensor $\ten{P}(\mat{x}_n) \in \mathbb{R}^{J_1 \times \cdots \times J_D}$ with a tensor of unknown latent variables $\ten{W} \in \mathbb{R}^{J_1 \times \cdots \times J_D}$, corrupted by measurement noise $\epsilon_n$. Tensors in this context are $D$-dimensional arrays, with vectors $(D=1)$ and matrices $(D=2)$ being the most common cases. Vectorizing all tensors and collecting the measurements $y_1,\ldots,y_N$ into a vector $\mat{y} \in \mathbb{R}^N$ allows~\eqref{eq:forwardtensormodel} to be rewritten into the linear system of equations
\begin{align}
\label{eq:forwardmodel}
    \mat{{y}} = \mat{\Phi}(\mat{x})\, \mat{w} + \mat{\epsilon} .
\end{align}
Row $n$ of the matrix $\mat{\Phi}(\mat{x}) \in \mathbb{R}^{N \times J_1\cdots J_D}$ contains the vectorization of the tensor $\ten{P}(\mat{x}_n)$. For notational convenience the indication that $\mat{\Phi}$ depends on $\mat{x}$ is dropped from here on. The inverse problem consists of inferring the latent variables $\mat{w}$ from the noisy measurements $\mat{y}$. Inverse problems of this kind appear in many different applications fields such as machine learning~\cite{blondel2016polynomial,stoudenmire2016supervised,LSSVMbook,wesel2021large,williams2006gaussian} control~\cite{batselier2022low,batselier2017tensor,pillonetto2010new,sarkka2023bayesian} and signal processing~\cite{epstein2007introduction,hansen2006deblurring,kargas2021supervised,ko2020fast,liu2012tensor,mastronardi2000fast,wahls2014learning}. In this article a Bayesian approach~\cite{bardsley2018computational} is considered by assuming that $\mat{w}$ and $\mat{\epsilon}$ are random variables. The goal is then to infer the posterior distribution $p(\mat{w}|\mat{{y}})$ of $\mat{w}$ conditioned on the measurements $\mat{y}$ using Bayes' theorem
\begin{align*}
    p(\mat{w} | \mat{{y}}) =  \frac{p(\mat{{y}} | \mat{w} ) \; p({\mat{w}}) }{p(\mat{{y}})}.
\end{align*}
The distribution $p(\mat{w})$ is called the prior and encodes a belief on what $\mat{w}$ is before the measurements are known. The main contribution of this article is the complete characterization of a prior $p(\mat{w})$ that encodes the belief that the corresponding tensor $\ten{W}$ is structured.
A Gaussian distribution is assumed for the noise distribution $p(\mat{\epsilon}) = \mathcal{N}(\bm{0},\mat{\Sigma} )$ with mean vector $\mat{0}$ and covariance matrix $\mat{\Sigma}$ and likewise for the prior $p(\mat{w})  = \mathcal{N}(\mat{w}_0, \mat{P}_0)$. The linear forward model~\eqref{eq:forwardmodel} combined with the Gaussian assumptions results in a Gaussian posterior $p(\mat{w}|\mat{y})=\mathcal{N}(\mat{w}_+,\mat{P}_+)$ with mean vector $\mat{w}_+$ and covariance matrix $\mat{P}_+$
\begin{align}
   \mat{w}_+ &= (\mat{P}_0^{-1} + \mat{\Phi}^T\mat{\Sigma}^{-1}\mat{\Phi})^{-1}\, (\mat{\Phi}^T\mat{\Sigma}^{-1}{\mat{y}}+\mat{P}_0^{-1}\mat{w}_0 ), \label{eq:meanposterior} \\
  \mat{P}_+ &= (\mat{P}_0^{-1} + \mat{\Phi}^T\mat{\Sigma}^{-1}\mat{\Phi})^{-1}. \label{eq:covposterior}
\end{align}
The role of the prior $p(\mat{w})$ can now be understood from~\eqref{eq:meanposterior} and~\eqref{eq:covposterior}. In the absence of data ($\mat{\Phi}=\mat{0}$ and $\mat{y}=\mat{0}$) the posterior equals the prior. In other words, the prior encodes a belief on what the solution $\mat{w}$ of~\eqref{eq:forwardmodel} should be before any data is known. A natural question to ask is then what kind of prior to use. In this article we consider a prior encoding the belief that the tensor $\ten{W}$ has a structure that is completely determined by a matrix $\mat{A} \in \mathbb{R}^{I \times J_1\cdots J_D}$ and vector $\mat{b} \in \mathbb{R}^I$ such that
\begin{align*}
    \mat{A} \; \vect{(\ten{W})} &= \mat{b},
\end{align*}
which we will refer to as $(\mat{A},\mat{b})$-constrained tensors.
The contributions of this article are threefold.
\begin{enumerate}
    \item We show how the definition of $(\mat{A},\mat{b})$-constrained tensors is well-motivated since it encompasses a wide variety of relevant structured tensors.  Examples are given for tensors with fixed entries, tensors with known sums of entries and symmetric, Hankel, Toeplitz, circulant, and triangular tensors.
    \item In Theorem~\ref{theo:nullspace} we completely characterize the mean vector $\mat{w}_0$ and covariance matrix $\mat{P}_0$ of the prior $p(\mat{w})$ for $(\mat{A},\mat{b})$-constrained tensors.
    \item In Theorems~\ref{theo:covar1} and~\ref{theo:orthcovar} we provide explicit expressions for $\mat{P}_0$ for $(\mat{A},\mat{b})$-constrained tensors whose entries remain invariant under a permutation $\mat{P}$. Such tensors will be called $\mat{P}$-invariant or skew-$\mat{P}$-invariant.
\end{enumerate}
These three contributions are important because the prior mean $\mat{w}_0$ and covariance matrix $\mat{P}_0$ are necessary to solve the Bayesian inverse problem via equations~\eqref{eq:meanposterior} and~\eqref{eq:covposterior}.
Contrary to most solution strategies for linear least squares problems the matrix inverse of $\mat{P}_0^{-1} + \mat{\Phi}^T\mat{\Sigma}^{-1}\mat{\Phi}$ is explicitly required as it forms the posterior covariance. Also note that the dimension of the matrix to invert is $J_1J_2\ldots J_D \times J_1J_2\ldots J_D$, which limits the use of direct solvers to cases of small $J$ and $D$. Hybrid projection methods~\cite{chung_silvia,chung2017} are a viable alternative for cases where $J$ and $D$ are prohibitively large. Another alternative is to solve the corresponding dual problem, which is described in terms of the so-called kernel matrix $\mat{\Phi}\,\mat{P}_0 \mat{\Phi}^T \in \mathbb{R}^{N \times N}$. This approach is commonly used in least-squares support vector machines~\cite{LSSVMbook} and Gaussian Processes~\cite{williams2006gaussian} and has a computational complexity of at least $O(N^2)$. When the tensor $\ten{P}(\mat{x}_n)$ exhibits a low-rank structure then another way to obtain low computational complexity of solving~\eqref{eq:meanposterior} is by imposing a low-rank tensor structure to $\mat{w}_+$ and $\mat{P}_+$~\cite{batselier2017tensor,novikov_exponential_2018,stoudenmire2016supervised}. Developing dedicated solution strategies for equations~\eqref{eq:meanposterior} and~\eqref{eq:covposterior}, however, lies outside the scope of this article. 



\subsection{Notation}
\label{subsec:notation}

Tensors in this article are multi-dimensional arrays with real entries. We denote scalars by italic letters $a, b, \ldots $, vectors by boldface italic letters $\mat{a}, \mat{b}, \ldots$, matrices by boldface capitalized italic letters $\mat{A}, \mat{B}, \ldots$ and higher-order tensors by boldface calligraphic italic letters $\ten{A}, \ten{B}, \ldots $. The vector $\mat{e}_{j_d} \in \mathbb{R}^{J_d}$ denotes a canonical basis vector that has a single nonzero unit entry at position $j_d$. The vector $\mat{1}_{J_d} \in \mathbb{R}^{J_d}$ denotes a vector of ones and $\mat{I}_{J_d} \in \mathbb{R}^{J_d \times J_d}$ is the unit matrix. The number of indices required to determine an entry of a tensor is called the order of the tensor. A $D$th order or $D$-way tensor is hence denoted $\ten{A} \in \mathbb{R}^{J_1 \times J_2 \times \cdots \times J_D}$. An index $j_d$ always satisfies $1\leq j_d \leq J_d$, where $J_d$ is called the dimension of that particular mode. Tensor entries are denoted $w_{j_1,j_2,\cdots, j_D}$. The merger of a set of $d$ separate indices $j_1,j_2,\ldots,j_d$ is denoted by the single index $$\overline{j_1 j_2 \ldots j_d}=j_1+(j_2-1)\,J_1 + \cdots + (j_d-1) J_1 \cdots J_{d-1}.$$
For a tensor $\ten{W}$ we will always assume that the corresponding vector $\mat{w} = \textrm{vec}(\ten{W})$. The square root matrix $\sqrt{\mat{P}}$ of $\mat{P}$ satisfies per definition $\mat{P}=\sqrt{\mat{P}}\,(\sqrt{\mat{P}})^T.$

\section{$(\pmb{A},\pmb{b})$-constrained tensors}
Before characterizing the prior $p(\mat{w})$ we first demonstrate the breadth of $(\mat{A},\mat{b})$-constrained tensors through eight examples. These examples demonstrate that the definition of $(\mat{A},\mat{b})$-constrained tensors is well-motivated in that it captures a wide variety of structured tensors.

\subsection{Tensors with fixed entries}
A tensor $\ten{W} \in \mathbb{R}^{J_1 \times J_2 \times \cdots \times J_D}$ with $I$ fixed entries can be described as $\mat{A} \, \mat{w} = \mat{b}$ where row $i$ of the matrix $\mat{A} \in \mathbb{R}^{I \times J_1\cdots J_D}$ is a canonical basis vector $\mat{e}_{\overline{j_1\cdots j_D}}$ that selects entry $w_{j_1,\ldots,j_D}$. The corresponding fixed numerical value of $w_{j_1,\ldots,j_D}$ is then given by $b_i$. Such fixed values are in practice usually zero, for example in triangular or banded matrices. Such structures can also be generalized to the tensor case.
\begin{definition}
\label{def:triangulartensor}
    A tensor $\ten{W} \in \mathbb{R}^{J_1 \times J_2 \times \cdots \times J_D}$ is lower (upper) triangular when $w_{j_1,j_2,\cdots,j_D}=0$ holds for each consecutive index pair $j_d,j_{d+1}$ such that $j_d-j_{d+1} < (>)\, 0$.
\end{definition}
The characterization of a lower (upper) triangular tensor as an $(\mat{A},\mat{b})$-constrained tensor is given in the following lemma.
\begin{lemma}
    \label{lemma:fixedentries}
Let $\mat{L}$ be the $J(J-1)/2 \times J^2$ matrix that has on each row a single unit entry for each particular occurrence of $j_1-j_2 <(>)\,0$. Lower (upper) triangular tensors are then described by
\begin{align*}
    \mat{A} &= \begin{pmatrix}
        \mat{L}  \otimes  \mat{I}_J  \otimes   \cdots  \otimes  \mat{I}_J\\
        \mat{I}_J \otimes  \mat{L} \otimes  \cdots \otimes  \mat{I}_J\\
        \vdots \\
        \mat{I}_J \otimes  \mat{I}_J \otimes  \cdots \otimes  \mat{L}
    \end{pmatrix} \in \mathbb{R}^{\frac{(D-1)(J-1)J^{D-1}}{2} \times J^D},
\end{align*}
and a vector $\mat{b} \in \mathbb{R}^{\frac{(D-1)(J-1)J^{D-1}}{2}}$ of zeros.
\end{lemma}
\begin{proof}
The known fixed values of lower (upper) triangular tensors are zero and hence $\mat{b}$ is a vector of zeros. Each row of the matrix $\mat{A}$ has a single unit entry to select a particular tensor entry for which some consecutive indices $j_d,j_{d+1}$ satisfy $j_d-j_{d-1} <(>)\, 0$. A tensor with $D$ indices has $D-1$ consecutive index pairs and therefore $\mat{A}$ is partitioned into $D-1$ block rows. Each block row is a Kronecker product of $D-2$ identity matrices with $\mat{L}$. The Kronecker product of identity matrices generates all possible index combinations of $D-2$ index values. The $\mat{L}$ matrix factor in the Kronecker product adds the remaining 2 indices but only considers index pairs for which  $j_d-j_{d-1} <(>)\, 0$.  
\end{proof}
The $\mat{A}$ matrix that describes tensors with known fixed entries in Lemma~\ref{lemma:fixedentries} is sparse and highly structured as demonstrated by the following example.
\begin{example}
\label{example:lowtriang}
Consider a lower triangular tensor $\ten{W} \in \mathbb{R}^{3 \times 3 \times 3}$. The condition $j_d-j_{d+1}<0$ occurs in 3 cases $(j_d,j_{d+1}) \in \{(1,2),(1,3),(2,3)\}$. Defining the matrix $\mat{L} \in \mathbb{R}^{3 \times 9}$ with 3 nonzero entries
\begin{align*}
    l_{1,\overline{12}} = l_{2,\overline{13}}= l_{3,\overline{23}}=1 
\end{align*}
allows us to describe the desired $\mat{A}$ matrix as
\begin{align}
\label{eq:triaA}
    \mat{A} &= \begin{pmatrix}
        \mat{I}_3 \otimes \mat{L} \\
        \mat{L} \otimes  \mat{I}_3
    \end{pmatrix} \in \mathbb{R}^{18 \times 27}.
\end{align}
This particular sparse structure is exploited in Section~\ref{sec:nullspace} when a basis for the nullspace of $\mat{A}$ needs to be computed. Note that there are actually only 17 zero entries for which $j_d-j_{d+1}<0$, which implies that the $\mat{A}$ matrix from equation~\eqref{eq:triaA} counts the case $j_1=1,j_2=2,j_3=3$ twice. This, however, does not negatively affect the resulting prior. 
\end{example}

\subsection{Known sum of entries}
Tensors for which the sum over all or only particular entries add up to a known value are also quite common in applications. Stochastic tensors are a particular example~\cite{gleich2015multilinear,li2014limiting}. Knowing a particular sum of entries can be described as follows.
\begin{lemma}
\label{lem:constantsum}
    Tensors $\ten{W}\in \mathbb{R}^{J_1 \times \cdots \times J_D}$ for which the sum over the entries of an index set $\mathcal{J}$ is a tensor $\ten{B}$ are described by
\begin{align}
\label{eq:Asum}
 \mat{A}\;\mat{w} = \vect{(\ten{B})} \; \textrm{ with } \;   \mat{A} = \mat{A}_D \otimes \cdots \otimes \mat{A}_1,  
\end{align}
where each matrix $\mat{A}_d \;(d=1,\ldots,D)$ in the Kronecker product is per definition
\begin{equation}
\label{eq:sumAs}
    \mat{A}_d = 
    \begin{cases}
        \mat{1}_{J_{j_d}}^T & \text{if } j_d \in \mathcal{J}, \\
        \mat{I}_{J_{j_d}}& \text{if } j_d \notin \mathcal{J}.
    \end{cases}
\end{equation}
\end{lemma}
The Kronecker product in~\eqref{eq:Asum} has as its leftmost factor $d=D$ and runs towards $d=1$ due to the opposite ordering of indices in the Kronecker product.
\begin{proof}
With the definitions of the $\mat{A}_d$ matrices the sum over the relevant entries of $\ten{W}$ is written in terms of n-mode products~\cite[p.~460]{kolda2009tensor}
\begin{align*}
\ten{W} \times_1 \mat{A}_1 \times_2 \cdots \times_D \mat{A}_D &= \ten{B}.
\end{align*}
Using the vectorization operation this can be rewritten as
\begin{align*}
    \left(\mat{A}_D \otimes \cdots \otimes \mat{A}_1 \right)\;\mat{w} &= \mat{b},
\end{align*}
which finalizes the proof.    
\end{proof}
\begin{example}
Let $\mat{W} \in \mathbb{R}^{2 \times 3}$ be a matrix for which each each row sum equals to 1. Lemma~\ref{lem:constantsum} then implies that
\begin{align*}
    \mat{A} &= \begin{pmatrix}1 & 1  & 1\end{pmatrix} \otimes  \begin{pmatrix}1 & 0\\0 & 1 \end{pmatrix}, \; \mat{b} = \mat{1}_2.
\end{align*}
\end{example}

\subsection{Eigenvector structure}
Tensors whose vectorization is an eigenvector of a matrix $\mat{P}$ with eigenvalue $\lambda$ are described by the constraint $\mat{A}=\lambda\,\mat{I}-\mat{P}$ and $\mat{b}=\mat{0}$. 
An important structure in this article is obtained when $\mat{P}$ is a permutation matrix. Indeed, $\mat{P}\,\mat{w}=\mat{w}$ then implies that the entries of $\ten{W}$ remain invariant under the permutation $\mat{P}$. The distinction between $\lambda=1$ and $\lambda=-1$ is made explicit through the following two definitions.
\begin{definition}
\label{def:Pinvar}
Let $\mat{P} \in \mathbb{R}^{J^D \times J^D}$ be a permutation matrix. A $\mat{P}$-invariant tensor $\ten{W}$ is defined by
\begin{align*}
    \left(\mat{I}-\mat{P} \right) \mat{w} &= \mat{0} \Leftrightarrow \mat{P}\, \mat{w} =  \mat{w}.
\end{align*}
Likewise, a skew-$\mat{P}$-invariant tensor $\ten{W}$ satisfies per definition 
\begin{align*}
    \left(-\mat{I}-\mat{P} \right) \mat{w} &= \mat{0} \Leftrightarrow \mat{P}\, \mat{w} =  -\mat{w}.
\end{align*}
\end{definition}
In this way any particular permutation matrix $\mat{P}$ then defines a corresponding structured tensor. 
Next we discuss some prominent examples of $\mat{P}$-invariant tensor structures.
\begin{definition}(Cyclic Symmetric tensor~\cite{batselier2017constructive})
\label{def:genpershuf}
The cyclic index shift permutation matrix $\mat{C}$ of a $D$-way tensor {$\ten{W}$} is the $J^D \times J^D$ permutation matrix 
\begin{align*}
\mat{C} \;=\; \begin{pmatrix}
\mat{I}(1:I^{D-1}:I^D,:)\\
\mat{I}(2:I^{D-1}:I^D,:)\\
\vdots\\
\mat{I}(I^{D-1}:I^{D-1}:I^D,:)\\
\end{pmatrix},
\end{align*}
where $\mat{I}$ is the $J^D \times J^D$ identity matrix and Matlab colon notation is used to denote submatrices. A $\mat{C}$-invariant tensor $\ten{W}$ is then called a cyclic symmetric tensor. 
\end{definition}
Defining the vector $\tilde{\mat{w}}:=\mat{C}\;\textrm{vec}(\ten{W})$ it can be verified that
\begin{align*}
    \tilde{w}_{j_D, j_1, \ldots j_{D-1}} &= w_{j_1,  \ldots j_{D-1}, j_D}.
\end{align*}
In other words, $\mat{C}$ performs a cyclic shift of the indices to the right.
When $D=2$, then $\mat{C}$ uniquely defines $J \times J$ symmetric matrices $\mat{W}$ since the cyclic index shift property implies that $\tilde{w}_{j_2,j_1}=w_{j_1,j_2}$~\cite{van2000ubiquitous}. The case $D > 2$ does not result in a fully symmetric tensor, as for example the required index permutation $j_1,j_2,j_3 \rightarrow j_1,j_3,j_2$ would not be enforced by $\mat{C}$. $\mat{C}$-invariance is therefore a weaker constraint than full symmetry.

\begin{definition}(Symmetric tensor) Let $\mat{S}$ be the permutation matrix such that all entries of $\tilde{\mat{w}}:=\mat{S}\;\textrm{vec}(\ten{W})$ satisfy $\tilde{w}_{j_1,\ldots,j_D}= w_{\pi(j_1,\ldots, j_D)}$, where $\pi(j_1,\ldots, j_D)$ is any permutation of the indices. A $\mat{S}$-invariant tensor $\ten{W}$ is per definition a symmetric tensor.
    
\end{definition}
\begin{definition}(Centrosymmetric tensor~\cite{batselier2017constructive})
A $\mat{J}$-invariant tensor $\ten{W}$, where $\mat{J}$ is the column-reversed identity matrix, is called a centrosymmetric tensor.
\end{definition}
A centrosymmetric tensor $\ten{W}$ satisfies
\begin{align*}
    w_{j_1,\ldots, j_D} = w_{J_1-j_1+1,\ldots,J_D-j_D+1}.
\end{align*}
Probably the most famous tensor that exhibits centrosymmetry is the matrix-matrix multiplication tensor~\cite{de1978varieties}.
\begin{definition}(Hankel Tensor)
Let $\mat{H} \in \mathbb{R}^{J^D \times J^D}$ be the permutation matrix that cyclically permutes all $D$ indices $j_1,\ldots,j_D$ with constant index sum $j_1+\cdots +j_D$. A $\mat{H}$-invariant tensor $\ten{W}$ is called a Hankel tensor. 
\end{definition}
The minimal index sum is $D=1+1+1+\cdots+1$ and maximal index sum is $JD=J+J+\cdots +J$. This implies that $\mat{H}$ consists of  $JD-D+1$ permutation cycles and $\textrm{rank}(\mat{H})=JD-D+1$.
\begin{definition}(Toeplitz Tensor)
Let $\mat{T} \in \mathbb{R}^{J^D \times J^D}$ be the permutation matrix that cyclically permutes all indices $j_d \mapsto j_d+1$, where $J_d + 1 \mapsto 1 \;(d=1,\ldots,D)$. A $\mat{T}$-invariant tensor $\ten{W}$ is called a Toeplitz tensor. 
\end{definition}
A special case of a Toeplitz tensor is a circulant tensor.
\begin{definition}(Circulant Tensor)
Let $\mat{T} \in \mathbb{R}^{J^D \times J^D}$ be the permutation matrix that cyclically permutes all indices $j_d \mapsto \bmod(j_d+1,J_d) \neq 0$. If $\bmod(j_d+1,J_d) = 0$ then $j_d \mapsto J_d \;(d=1,\ldots,D)$. A $\mat{T}$-invariant tensor $\ten{W}$ is called a circulant tensor. 
\end{definition}

\section{Full characterization of the prior distribution}
\label{sec:nullspace}
In this section the Gaussian prior $p(\mat{w})$ for $(\mat{A},\mat{b})$-constrained tensors is fully characterized. We also discuss how the square root covariance matrix $\sqrt{\mat{P}}_0$ can be computed without explicitly constructing the matrix $\mat{A}$ through a block-row partitioning of $\mat{A}$.

\begin{theorem}
\label{theo:nullspace}
The Gaussian distribution of $(\mat{A},\mat{b})$-constrained tensors $\mathcal{N}(\mat{w}_0, \mat{P}_0)$ is described by a mean vector $\mat{w}_0$ such that $\mat{A}\,\mat{w}_0=\mat{b}$ and by a covariance matrix $\mat{P}_0$ such that the columns of $\sqrt{\mat{P}_0}$ span the right nullspace of $\mat{A}$.
\end{theorem}
\begin{proof}
Let $\mat{x} \in \mathbb{R}^{J_1\ldots J_D}$ be a sample of the standard normal distribution $\mathcal{N}(\mat{0},\mat{I})$. A sample $\mat{w}$ of the desired Gaussian distribution is then
\begin{align*}
        \mat{w} &= \mat{w}_0 + \sqrt{\mat{P}_0}\;\mat{x},
\end{align*}
where $\sqrt{\mat{P}_0}$ is the matrix square root of the covariance matrix $\mat{P}_0$. Any sample $\mat{w}$ being an $(\mat{A},\mat{b})$-constrained tensor implies
\begin{align}
\label{eq:constraint}
    \mat{A}\;\mat{w} =     \mat{A}\;\mat{w}_0 + \mat{A}\;\sqrt{\mat{P}_0}\;\mat{x} &= \mat{b}.\end{align}
Equation~\eqref{eq:constraint} can only be true for all random samples $\mat{x}$ if and only if
\begin{align*}
    \mat{A}\;\mat{w}_0 &= \mat{b}, \\
    \mat{A}\;\sqrt{\mat{P}_0} &= \mat{0}.
\end{align*}
In other words, the mean $\mat{w}_0$ of the prior also has to satisfy the linear constraint and the columns of $\sqrt{\mat{P}_0}$ span the right nullspace of $\mat{A}$.    
\end{proof}



\subsection{Recursive nullspace computation}
When the matrix $\mat{A}$ is too large to construct explicitly then it is beneficial to compute a basis for its right nullspace recursively. This is possible when considering a partitioning into $S$ block-rows $\mat{A}=\begin{pmatrix} \mat{A}_1^T & \mat{A}_2^T & \hdots & \mat{A}_S^T
    \end{pmatrix}^T.$
Algorithm~\ref{alg:multipleconstraints} recursively computes a basis for this nullspace without ever explicitly constructing $\mat{A}$ using Theorem 6.4.1 from~\cite[p.~329]{golub2013matrix}.
\begin{algorithm}
\caption{Compute basis for nullspace ${\mat{V}_2}$ for block-row partitioned $\mat{A}$ matrix}
\label{alg:multipleconstraints}
\begin{algorithmic}
\REQUIRE $\mat{A}_1,\mat{A}_2,\ldots,\mat{A}_S$
\STATE{$\mat{V}_2 \gets \textrm{null}(\mat{A}_1)$}
\FOR{$s=2:S$}
\STATE{$\mat{Z}_{s} \gets \textrm{null}(\mat{A}_s\,\mat{V}_2)$}
\STATE{$\mat{V}_2 \gets \mat{V}_2\,\mat{Z}_s$}
\ENDFOR
\RETURN $\mat{V}_2$
\end{algorithmic}
\end{algorithm}

\section{Explicit covariance matrix construction for permutation-invariant tensors}
\label{sec:explicit}
Computing the covariance matrix $\mat{P}_0$ via Theorem~\ref{theo:nullspace} requires a basis for the nullspace of $\mat{A}$. For $\mat{P}$-invariant tensors it is possible to derive an explicit formula for $\mat{P}_0$ as a function of the permutation matrix $\mat{P}$, which enables efficient sampling of the prior. Before we can state the main result in Theorem~\ref{theo:covar1}, we first need to discuss some facts about permutation matrices. An important concept tied to permutation matrices is its order. Any permutation can be written as a product of disjoint cycles. Each cycle has a particular length, also called the order of the cycle. In this article $K$ will denote the least common multiple of all orders of disjoint cycles of a given permutation.
\begin{definition}
The order $K \in \mathbb{N}$ of a permutation matrix $\mat{P}$ is defined as the smallest natural number such that $\mat{P}^K=\mat{I}$.   
\end{definition}
Skew-$\mat{P}$-invariant structures always have an even order $K$.
\begin{lemma}
\label{lemma:skewPeven}
A skew-$\mat{P}$-invariant structure has an even order $K$.    
\end{lemma}
\begin{proof}
Skew-$\mat{P}$-invariance requires per definition that $\lambda=-1$. From $\mat{P}^K\,\mat{w} = \mat{I}\,\mat{w} =  (-1)^K \, \mat{w}$ it follows that $(-1)^K=1$, which proves the desired.
\end{proof}
Theorem~\ref{theo:covar1} will express the desired covariance matrix $\mat{P}_0$ as a function of powers of the permutation matrix $\mat{P}$. 
The following two lemmas relating powers of permutation matrices are easily proved.
\begin{lemma}
\label{cor:powersofP}
Let $\mat{P}$ be a permutation matrix of order $K$, then for any $1 \leq k \leq K $:
\begin{align}
    \mat{P}^k &= \mat{P}^{K+k} .
\end{align}
\end{lemma}
\begin{lemma}
\label{cor:Ptransp}
Let $\mat{P}$ be a permutation matrix of order $K$, then for any $1 \leq k \leq K $:\begin{align}
    \mat{P}^{K-k} &= \left( \mat{P}^{k} \right)^T.
\end{align}
\end{lemma}
Lemma~\ref{cor:powersofP} follows from $\mat{P}^K=\mat{I}$. Lemma~\ref{cor:Ptransp} follows from the orthogonality of permutation matrices and from the fact that powers of permutation matrices are still permutation matrices. We now have all ingredients to describe the main result that provides an analytic solution for the covariance matrix $\mat{P}_0$ as an average over powers of the permutation matrix $P$.
\begin{theorem}
\label{theo:covar1}
Let $\mat{P}$ be a permutation matrix of order $K$. The Gaussian distribution of $\mat{P}$-invariant tensors $\mathcal{N}(\mat{w}_0,\mat{P}_0)$ is described by a mean vector $\mat{w}_0$ that is $\mat{P}$-invariant and covariance matrix 
\begin{align}
    \label{eq:covar1}
    \mat{P}_0 &= \frac{\mat{P} + \mat{P}^2 + \cdots +\mat{P}^K}{K}.
\end{align}
\end{theorem}
The $\mat{P}$-invariance of the mean $\mat{w}_0$ follows directly from Theorem~\ref{theo:nullspace}. The proof of Theorem~\ref{theo:covar1} therefore requires showing that $\mat{P}_0$ in~\eqref{eq:covar1} is the desired covariance matrix. A matrix $\mat{P}_0$ is a covariance matrix if it satisfies the following three sufficient conditions:
\begin{enumerate}
    \item has positive diagonal entries,
    \item is symmetric,
    \item is positive (semi-)definite.
\end{enumerate}
Short proofs will now be given for each of these three covariance conditions.
\begin{lemma}
\label{lemma:posdiag}
The matrix $\mat{P}_0$ has positive diagonal entries.    
\end{lemma}
\begin{proof}
$\mat{P}_0$ is per definition a sum of permutation matrices, all diagonal entries of $\mat{P}_0$ are therefore either zero or positive. Since $\mat{P}^K=\mat{I}$ we have that the diagonal entries are guaranteed to be positive.    
\end{proof}
\begin{lemma}
The matrix $\mat{P}_0$ is symmetric.    
\end{lemma}
\begin{proof}
The symmetry of $\mat{P}_0$ follows from
\begin{align*}
    \mat{P}_0^T &= \frac{\mat{P}^T + (\mat{P}^2)^T + \cdots + (\mat{P}^{K-1})^T  + (\mat{P}^K)^T}{K}, \\
    &= \frac{\mat{P}^{K-1} + \mat{P}^{K-2} + \cdots + \mat{P} + \mat{P}^K}{K}, \\
    &= \mat{\mat{P}}_0,
\end{align*}
where the second line follows from Lemma~\ref{cor:Ptransp}.    
\end{proof}
The semi-positive definiteness of $\mat{P}_0$ follows from its idempotency.
\begin{lemma}
The matrix $\mat{P}_0$ is idempotent, that is $\mat{P}_0^2 = \mat{P}_0$.
\end{lemma}
\begin{proof}
Writing out $(K\mat{P}_0)^2$ in terms of $\mat{P}$ and applying Lemma~\ref{cor:powersofP} results in
 \begin{align*}
 &(\mat{P} + \mat{P}^2 + \cdots + \mat{P}^K)^2,  \\
     &= \mat{P}^2 + 2\;\mat{P}^3 + \cdots +(K-1)\;\mat{P}^{K} + K\;\mat{P}^{K+1} + (K-1)\;\mat{P}^{K+2}+\cdots+2\;\mat{P}^{2K-1}+\mat{P}^{2K}, \\
     &= K\;\mat{P} +\underbrace{\mat{P}^2 + (K-1)\;\mat{P}^{K+2} }_{K\;\mat{P}^{2}} + \cdots + \underbrace{2\;\mat{P}^{2K-1} + (K-2)\;\mat{P}^{K-1} }_{K\;\mat{P}^{K-1}} + \underbrace{(K-1)\;\mat{P}^K+\;\mat{P}^{2K}}_{K\mat{P}^K}, \\
     &= K\;(\mat{P}+\mat{P}^2+\mat{P}^3+\cdots+\mat{P}^K),\\
     &= K^2\;\mat{P}_0,
 \end{align*}    
 which proves that $\mat{P}_0$ is idempotent.
\end{proof}
The first consequence of $\mat{P}_0$ being idempotent is that it is positive semi-definite.
\begin{lemma}
The matrix $\mat{P}_0$ is positive semi-definite.    
\end{lemma}
\begin{proof}
The two eigenvalue equations
\begin{align*}
    \mat{P}_0 \, \mat{v} =  \lambda \, \mat{v} \quad ,  \quad  (\mat{P}_0)^2 \, \mat{v} =  \lambda^2 \, \mat{v}
\end{align*}
are actually equal due to $\mat{P}_0$ being idempotent. It therefore follows that $\lambda^2-\lambda=0$, which implies that the eigenvalues are either 0 or 1. This proves the positive semi-definiteness of $\mat{P}_0$.
\end{proof}
Having proved that $\mat{P}_0$ is a covariance matrix it remains to show that samples drawn from $\mathcal{N}(\mat{w}_0,\mat{P}_0)$ are $\mat{P}$-invariant. From its symmetry and idempotency it follows that $\mat{P}_0$ is its own matrix square root
$\mat{P}_0 = \sqrt{\mat{P}_0} = \mat{P}_0^T = \sqrt{\mat{P}_0}^T$.
\begin{lemma}
\label{lemma:Pinvariance}
Every sample $\mat{w}$ drawn from $\mathcal{N}(\mat{w}_0,\mat{P}_0)$ is $\mat{P}$-invariant.
\end{lemma}
\begin{proof}
A sample $\mat{w}$ from $\mathcal{N}(\mat{w}_0,\mat{P}_0)$ can be drawn by computing
\begin{align*}
    \mat{w} = \mat{w}_0 + \sqrt{\mat{P}_0} \; \mat{x},
\end{align*}
where $\mat{x}$ is drawn from a standard normal distribution $\mathcal{N}(\mat{0},\mat{I})$. The $\mat{P}$-invariance of $\mat{w}$ follows from
\begin{align*}
    \mat{w} &= \mat{P}\;\mat{w}, \\
  \mat{w}_0 + \sqrt{\mat{P}_0} \; \mat{x} &= \mat{P}\;\mat{w}_0 + \mat{P}\;\sqrt{\mat{P}_0} \; \mat{x},\\
   \sqrt{\mat{P}_0} \; \mat{x} &= \mat{P}\;\sqrt{\mat{P}_0} \; \mat{x},\\
 \sqrt{\mat{P}_0} \; \mat{x} &= \mat{P}\;\left(\frac{\mat{P}+\mat{P}^2+\cdots+\mat{P}^{K-1}+\mat{P}^{K}}{K}\right) \; \mat{x},\\
  \sqrt{\mat{P}_0} \; \mat{x} &= \left(\frac{\mat{P}^2+\mat{P}^3+\cdots+\mat{P}^{K}+\mat{P}}{K}\right) \; \mat{x},\\
 \sqrt{\mat{P}_0} \; \mat{x} &= \sqrt{\mat{P}_0} \; \mat{x}.
\end{align*}
The terms that depend on $\mat{w}_0$ cancel due to the $\mat{P}$-invariance of $\mat{w}_0$. Lemma~\ref{cor:powersofP} is used to go from line 4 to line 5.
\end{proof} 
Lemmas~\ref{lemma:posdiag} up to~\ref{lemma:Pinvariance} constitute the proof of Theorem~\ref{theo:covar1}. Another consequence from the idempotency of $\mat{P}_0$ is that this matrix is its own pseudoinverse.
\begin{lemma}
The pseudoinverse $\mat{P}_0^{\dagger}$ satisfies 
\begin{align*}
    \mat{P}_0^{\dagger} &= \mat{P}_0.
\end{align*}    
\end{lemma}
\begin{proof}
The pseudoinverse $\mat{P}_0^{\dagger}$ needs to satisfy the following four properties:
\begin{enumerate}
    \item $\mat{P}_0\mat{P}_0^{\dagger}\mat{P}_0 = \mat{P}_0$,
    \item $\mat{P}_0^{\dagger}\mat{P}_0\mat{P}_0^{\dagger} = \mat{P}_0^{\dagger}$,
    \item $(\mat{P}_0\mat{P}_0^{\dagger})^T = \mat{P}_0\mat{P}_0^{\dagger}$,
    \item $(\mat{P}_0^{\dagger}\mat{P}_0)^T = \mat{P}_0^{\dagger}\mat{P}_0$.
\end{enumerate}
All these properties are satisfied when assuming $\mat{P}_0^{\dagger} = \mat{P}_0$ and they follow from the idempotency of $\mat{P}_0$. For example, Properties 1 and 2 follow from
\begin{align*}
\mat{P}_0\mat{P}_0^{\dagger}\mat{P}_0 &=\mat{P}_0^{\dagger}\mat{P}_0\mat{P}_0^{\dagger} = (\mat{P}_0)^3 = \mat{P}_0 = \mat{P}_0^{\dagger}.
\end{align*}
Properties 3 and 4 follow from the symmetry of $\mat{P}_0$.
\end{proof}
The fact that $\mat{P}_0 = \sqrt{\mat{P}_0} = \mat{P}_0^{\dagger} = \sqrt{\mat{P}_0^{\dagger}}$ is convenient for several reasons. First, no explicit $\mat{P}_0^{-1}$ computation is required in equations~\eqref{eq:meanposterior} and~\eqref{eq:covposterior}. Second, sampling $\mathcal{N}(\mat{w}_0,\mat{P}_0)$ can be done without a matrix square-root computation and without any matrix-vector multiplications. Using Theorem~\ref{theo:covar1} the product $\sqrt{\mat{P}_0}\,\mat{x}=\mat{P}_0\,\mat{x}$ can be implemented as a weighted sum of permuted versions of $\mat{x}$
\begin{align*}
\frac{\mat{P}\,\mat{x}+\mat{P}^2\,\mat{x} + \cdots + \mat{P}^K\,\mat{x}}{K}.    
\end{align*}
All information of the permutation $\mat{P}$ is contained in a vector $\mat{p}$ of $I^D$ elements that specifies how each entry gets mapped to the next. Each term $\mat{P}^k\,\mat{x}$ of the weighted sum is then computed by successive permutations of $\mat{x}$ according to $\mat{p}$ with computational complexity $O(I^D)$. The pseudocode for sampling the distribution is given in Algorithm~\ref{alg:fastsampling}. 
\begin{algorithm}
\caption{Generate $\mat{P}$-invariant sample from $\mathcal{N}(\mat{w}_0,\mat{P}_0)$}
\label{alg:fastsampling}
\begin{algorithmic}
\REQUIRE $\mat{w}_0$, index permutation vector $\mat{p}$
\STATE{$\mat{x} \gets \textrm{randn}(I^D)$} \hfill \% sample standard normal $\mathcal{N}(\mat{0},\mat{I})$
\STATE{$\mat{w} \gets K\,\mat{w}_0$}
\FOR{$k=1:K$}
\STATE{$\mat{w} \gets \mat{w} + \mat{x}$}
\STATE{$\mat{x} \gets \mat{x}[\mat{p}]$} \hfill \% permute entries of $\mat{x}$ according to $\mat{p}$
\ENDFOR
\STATE{$\mat{w} \gets \frac{\mat{w}}{K}$}
\RETURN $\mat{w}$
\end{algorithmic}
\end{algorithm}

A similar result as in Theorem~\ref{theo:covar1} can be proven for $\mat{P}$-skew-invariant tensors.
\begin{theorem}
\label{theo:skewcovar}
For a permutation of even order $K$, the Gaussian distribution of $\mat{P}$-skew-invariant tensors $\mathcal{N}(\mat{w}_0,\mat{P}_0)$ is described by a mean vector $\mat{w}_0$ that is $\mat{P}$-skew-invariant and covariance matrix
\begin{align}
    \label{eq:skewcovar}
    \mat{P}_0 &:= \frac{-\mat{P} + \mat{P}^2 - \cdots +\mat{P}^K}{K} = \frac{\sum_{k=1}^K\;(-1)^k \,\mat{P}^{k}}{K}.
\end{align}
\end{theorem}
\begin{proof}
The proof is very similar to that of Theorem~\ref{theo:covar1}. The diagonal entries being nonnegative can be derived from the following argument. The permutation matrix $\mat{P}$ itself consists of cyclic permutations, with either even or odd order. If a cyclic permutation has an even order $k$, then $\mat{P}^k$ will have ones on the diagonal for elements of the cycle. This cycle will occur $K/k$ times in~\eqref{eq:skewcovar}, always with a positive sign. If a cyclic permutation has odd order $k$, then the diagonal entries of $\mat{P}^k$ will come in equal amounts of $K/(2k)$ negative and $K/(2k)$ positive contributions, which results in a zero contribution to the diagonal. The total effect of all cyclic permutations then add up to either zero or positive diagonal entries.
Symmetry is proven by using Corollary~\ref{cor:Ptransp} and the fact that $K$ is even: an even order $k$ gets mapped to another even order $K-k$ and an odd order $k$ gets mapped to and odd order $K-k$. Hence, 
    \begin{align*}
        \mat{P}_0^T &= \frac{\sum_{k=1}^K\;(-1)^k \,(\mat{P}^{k})^T}{K} = \frac{\sum_{k=1}^K\;(-1)^k \,\mat{P}^{K-k}}{K} =\mat{P}_0.
    \end{align*}

The idempotency of $\mat{P}_0$ follows a similar proof as for the case of $\mat{P}$-invariance. Writing out $(K\mat{P}_0)^2$ in terms of $\mat{P}$ and applying Corollary~\ref{cor:powersofP} results in
 \begin{align*}
     &(-\mat{P} + \mat{P}^2 - \cdots + \mat{P}^K)^2  \\
     &= \mat{P}^2 - 2\;\mat{P}^3 + \cdots +(K-1)\;\mat{P}^{K} - K\;\mat{P}^{K+1} + (K-1)\;\mat{P}^{K+2} - \cdots - 2\;\mat{P}^{2K-1}+ \mat{P}^{2K} \\
     &= -K\;\mat{P} +\underbrace{\mat{P}^2 + (K-1)\;\mat{P}^{K+2} }_{K\;\mat{P}^{2}} - \cdots \underbrace{- 2\;\mat{P}^{2K-1} - (K-2)\;\mat{P}^{K-1} }_{-K\;\mat{P}^{K-1}} + \underbrace{(K-1)\;\mat{P}^K+\;\mat{P}^{2K}}_{K\mat{P}^K} \\
     &= K\;(-\mat{P}+\mat{P}^2-\mat{P}^3+\cdots+\mat{P}^K)\\
     &= K^2\;\mat{P}_0
 \end{align*}    
 which proves that $\mat{P}_0$ is idempotent.
\end{proof}

Theorems \ref{theo:covar1} and \ref{theo:skewcovar} are practical when the order $K$ of the permutation matrix $\mat{P}$ stays small compared to $J$ and $D$. For Hankel structures this is unfortunately not the case. Consider for example a $20 \times 20$ Hankel matrix. Its corresponding permutation matrix has permutation cycles ranging from length 1 up to 20 and $K$ is therefore the least common multiple of $1,2,\ldots,20 =232,792,560$. Fortunately, it is possible to explicitly construct a sparse matrix of orthogonal columns $\mat{V}$ such that $\sqrt{\mat{P}_0}=\mat{V}$. 

\section{Sparse square root covariance matrix construction for permutation-invariant tensors}
Every permutation $\mat{P}$ can be decomposed in terms of $R$ cyclic permutations. These cyclic permutations partition the set of all tensor entries into $R$ disjoint sets and allow for an alternative construction of $\sqrt{\mat{P}_0}$, where the resulting matrix is sparse and consists of orthogonal columns. 

\begin{theorem}
\label{theo:orthcovar}
Let $\mat{P}$ be a permutation matrix that consists of $R$ permutation cycles and let $C_r$ denote the $r$th cycle, where the number of tensor entries in $C_r$ is denoted $|C_r|$. Then the range of the matrix $\mat{V} \in \mathbb{R}^{J^D\times R}$ such that
\begin{equation}
v_{\overline{j_1,j_2,\ldots, j_D},r} = 
        \begin{cases}
        \frac{1}{\sqrt{|C_r|}} & \text{if  } w_{j_1,j_2,\ldots,j_D} \in {C}_r, \\[2ex]
        0& \text{otherwise, }
    \end{cases}
\end{equation}
spans the eigenspace of $\mat{P}$ corresponding to an eigenvalue $\lambda=1$. In other words, $\mat{V} = \sqrt{\mat{P}_0}$. Also, $\mat{V}^T\mat{V}=\mat{I}_R$.
\end{theorem}
\begin{proof}
The equality $\mat{P}\mat{V}=\mat{V}$ follows from each column of $\mat{V}$ containing nonzero values at tensor entries of a particular permutation cycle of $\mat{P}$. The orthogonality follows directly from the permutation cycles being disjoint and each column of $\mat{V}$ being unit-norm due to the scaling with $\sqrt{|C_r|}$.
\end{proof}
A basis for the skew-$\mat{P}$-invariant eigenspace can be built in a similar way by retaining the cycles of even order and alternating the sign of the entries $v_{\overline{j_1,j_2,\ldots, j_D},r}$ in each column. 

\begin{example}
Consider a $20 \times 20$ Hankel matrix. Using Theorem~\ref{theo:covar1} one would need to construct the $400 \times 400$ Hankel permutation matrix $\mat{H}$ and construct $\mat{P}_0$ by adding $232,792,560$ terms together. Using Theorem~\ref{theo:orthcovar} the sparse $400 \times 39$ matrix $\mat{V}$ can be constructed directly containing 400 nonzero entries.
\end{example}
\section{Solving the inverse problem}
In this section three different aspects when solving the inverse problem are discussed. First, we explain how the prior covariance matrices of $(\mat{A},\mat{b})$-constrained tensors can be parameterized. Second, we briefly discuss a change of variables, originally proposed in~\cite{chung2017}, to exploit fast implementations of the matrix vector product $\mat{P}_0\mat{w}$. The third aspect relates to kernel methods, where $(\mat{A},\mat{b})$-constrained tensor priors are used to define new structured tensor kernel functions.
\subsection{Parameterizing the prior covariance matrix}
\label{sec:parameterization}
The covariance matrix $\mat{P}_0$ as described in Theorems~\ref{theo:nullspace},~\ref{theo:covar1} and \ref{theo:orthcovar} encodes the structure of the $(\mat{A},\mat{b})$-constrained tensor without having any free parameters to quantify the importance of the prior $p(\mat{w})$ relative to the likelihood $p(\mat{y}|\mat{w})$. Such free parameters are often called hyperparameters. Suppose for example that through Theorem~\ref{theo:nullspace} an orthogonal basis for the nullspace $\mat{V}_2 \in \mathbb{R}^{J^D \times R}$ of $\mat{A}$ is computed from its singular value decomposition (SVD)
\begin{align*}
    \mat{A} = \begin{pmatrix}
        \mat{U}_1 & \mat{U}_2
    \end{pmatrix}\;\begin{pmatrix}
        \mat{S} & \mat{0}\\
        \mat{0} & \mat{0}\\
    \end{pmatrix}\;\begin{pmatrix}
        \mat{V}_1^T \\ \mat{V}_2^T
    \end{pmatrix}.
\end{align*}
A desired square-root covariance matrix $\sqrt{\mat{P}_0}$ is then $\mat{V}_2\,\mat{T}$, where $\mat{T} \in \mathbb{R}^{R \times R}$ is any invertible matrix.
The nullity $R$ of $\mat{A}$ can be interpreted as the total number of distinct elements in the $(\mat{A},\mat{b})$-constrained tensor $\ten{W}$.
The $\mat{T}$ matrix can be interpreted as the square-root covariance matrix of those $R$ variables since
\begin{align*}
    \mat{P}_0&= \sqrt{\mat{P}_0}\; (\sqrt{\mat{P}_0})^T = \mat{V}_2\;\left(\mat{T}\,\mat{T}^T \right) \; \mat{V}_2^T.
\end{align*}
The matrix $\mat{V}_2$ is then to be understood as ``projecting" the covariance matrix $\mat{T}\mat{T}^T$ of the $R$ underlying variables to the $J^D$ entries of the $(\mat{A},\mat{b})$-constrained $\ten{W}$ tensor. Parameterizing $\mat{T}$ in terms of a single hyperparameter $\sigma \in \mathbb{R^{+}}$ as $\mat{T} = \sigma \; \mat{I}$ implies that these $R$ variables are independent and have equal variance $\sigma^2$. Correlations between the $R$ variables can be modeled by for example parameterizing $\mat{T}$ as a lower triangular matrix. The values of these hyperparameters can be learned from data through cross-validation, marginal likelihood optimization or a hierarchical Bayesian approach~\cite{LSSVMbook,williams2006gaussian}.

\subsection{Change of variables}
Squaring the condition number when solving the normal equation of~\eqref{eq:meanposterior} can be avoided by solving its square-root version
\begin{align*}
    \begin{pmatrix}
        \sqrt{\mat{\Sigma}^{-1}}\mat{\Phi}\\
        \sqrt{\mat{P}_0^{-1}}
    \end{pmatrix} \; \mat{w}_+ &=   \begin{pmatrix}
        \sqrt{\mat{\Sigma}^{-1}}{\mat{y}}\\
        \sqrt{\mat{P}_0^{-1}}\,\mat{w}_0
    \end{pmatrix}
\end{align*}
instead. When constructing the square-root of the inverse prior covariance matrix is difficult then a change of variables can be used to avoid their construction~\cite{chung2017}. By defining $\mat{x}:=\mat{P}_0^{-1}\,(\mat{w}_+-\mat{w}_0)$ and $\mat{z}:=\mat{y}-\mat{\Phi}\mat{w}_0$ the square-root linear system is transformed into
\begin{align*}
    \begin{pmatrix}
        \sqrt{\mat{\Sigma}^{-1}}\mat{\Phi} \mat{P}_0\\
        \mat{I}
    \end{pmatrix} \; \mat{x} &=   \begin{pmatrix}
        \sqrt{\mat{\Sigma}^{-1}}{\mat{z}}\\
        0
    \end{pmatrix}.
\end{align*}
The desired posterior mean $\mat{w}_+$ can then be recovered from $\mat{w}_+=\mat{P}_0\,\mat{x}+\mat{w}_0$. This formulation is especially beneficial when the matrix vector product $\mat{P}_0\,\mat{x}$ can be implemented in a computationally efficient manner, for example using Algorithm~\ref{alg:fastsampling}.

\subsection{Structured tensor kernel functions}
When the tensor $\ten{W}$ is much larger than the data size $N$ then the $O(J^{3D})$ computational complexity of computing~\eqref{eq:meanposterior} is replaced with at least $O(N^2)$ by solving the corresponding dual problem
\begin{align*}
    (\mat{\Phi}\,\mat{P}_0 \, \mat{\Phi}^T + \mat{\Sigma} )\; \mat{v} = \mat{y}.
\end{align*}
An additional benefit is that no matrix inverse of $\mat{P}_0$ is required so that Theorems~\ref{theo:nullspace},~\ref{theo:covar1} and~\ref{theo:orthcovar} can be applied directly. The matrix $\mat{\Phi}\,\mat{P}_0 \, \mat{\Phi}^T$ is called the kernel matrix $\mat{K}$ and each entry $k_{i,j}$ is per definition the evaluation of a kernel function
$$
k_{i,j} = k(\mat{x}_i,\mat{x}_j) := \mat{\varphi}(\mat{x}_i)^T\,\mat{P}_0 \;\mat{\varphi}(\mat{x}_j).
$$
Choosing $\mat{P}_0$ as a covariance matrix of an $(\mat{A},\mat{b})$-constrained tensor allows us to define new kernel functions. The kernel trick in machine learning refers to the fact where the kernel function can be evaluated without every explicitly computing the possibly large feature vectors $\mat{\varphi}(\cdot)$. In the case of $\mat{P}$-invariant tensors one can exploit the particular structure of $\mat{P}_0$ as described in Theorem~\ref{theo:covar1} or use Algorithm~\ref{alg:fastsampling} to achieve this goal.
\begin{example}(Centrosymmetric polynomial kernel)
Let $\sqrt{c} \in \mathbb{R}$ and $d \in \mathbb{N}$. The polynomial kernel function is defined as
\begin{align*}
k(\mat{x}_i, \mat{x}_j ) &= \mat{\varphi}(\mat{x}_i)^T \;\mat{I}\; \mat{\varphi}(\mat{x}_j) ,\\
&= \underbrace{\begin{pmatrix}\sqrt{c} & \mat{x}_i^T \end{pmatrix} \otimes \cdots \otimes \begin{pmatrix}\sqrt{c} & \mat{x}_i^T \end{pmatrix}}_{d \textrm{ times}}\; \mat{I} \; \underbrace{\begin{pmatrix}\sqrt{c} & \mat{x}_j^T \end{pmatrix}^T \otimes \cdots \otimes \begin{pmatrix}\sqrt{c} & \mat{x}_j^T \end{pmatrix}^T}_{d \textrm{ times}}\\
&=(c + \mat{x}_i^T\,\mat{x}_j)^d.
\end{align*}
The expression $(c + \mat{x}_i^T\,\mat{x}_j)^d$ is obtained from writing the identity matrix $\mat{I}$ as a Kronecker product of smaller identity matrices and applying the mixed product property. The polynomial kernel function can therefore be interpreted as using a unit covariance matrix $\mat{P}_0$. We can now define the centrosymmetric polynomial kernel function $k_2$ by using the polynomial feature vectors $\mat{\varphi}(\cdot)$ and replacing $\mat{I}$ with the covariance matrix of centrosymmetric tensors. From Theorem~\ref{theo:covar1} it then follows that
\begin{align*}
k_2(\mat{x}_i, \mat{x}_j ) &= \mat{\varphi}(\mat{x}_i)^T \;\mat{P}_0\; \mat{\varphi}(\mat{x}_j) ,\\
&= \frac{1}{2}\,\mat{\varphi}(\mat{x}_i)^T \;{(\mat{I} + \mat{J} )}\; \mat{\varphi}(\mat{x}_j) ,\\
&= \frac{1}{2}\, \underbrace{\begin{pmatrix}\sqrt{c} & \mat{x}_i^T \end{pmatrix} \otimes \cdots \otimes \begin{pmatrix}\sqrt{c} & \mat{x}_i^T \end{pmatrix}}_{d \textrm{ times}}\; (\mat{I}+\mat{J}) \; \underbrace{\begin{pmatrix}\sqrt{c} & \mat{x}_j^T \end{pmatrix}^T \otimes \cdots \otimes \begin{pmatrix}\sqrt{c} & \mat{x}_j^T \end{pmatrix}^T}_{d \textrm{ times}},\\
&= \frac{1}{2}  (c+\mat{x}_i^T\mat{x}_j )^d + \frac{1}{2} \left(\begin{pmatrix} \sqrt{c} & \mat{x}_i^T \end{pmatrix} \mat{J}_d \begin{pmatrix} \sqrt{c} & \mat{x}_j^T \end{pmatrix}^T \right)^d. 
\end{align*}
Also here the explicit construction of $\mat{\varphi}(\cdot)$ is avoided by writing the matrix $\mat{J} \in  \mathbb{R}^{J^D \times J^D}$ as a Kronecker product of the smaller permutation matrix $\mat{J}_d \in \mathbb{R}^{J \times J}$ with itself $d$ times and using the mixed-product property.
\end{example}

\section{Applications}
In this section we demonstrate the use of Theorems~\ref{theo:nullspace},~\ref{theo:covar1}, and~\ref{theo:orthcovar} in three different applications. 
Practical implementations on how to sample various $(\mat{A},\mat{b})$-constrained tensor priors are explained in Application~\ref{exp:exp1}. We consider lower triangular tensors, tensors for which the sum over the last index adds up to 1, symmetric tensors and Hankel tensors. Application~\ref{exp:exp2} considers the problem of completing a Hankel matrix from noisy partial measurements by solving it as a Bayesian inverse problem. The estimate of the completed Hankel matrix when using a Hankel prior is compared to the estimate where no prior is used. In Application~\ref{ex:ex3} learning a classifier for handwritten digits is solved as a Bayesian inverse problem. The classifier obtained with the commonly used Tikhonov prior is compared to several $(\mat{A},\mat{b})$-constrained tensor priors.

All applications have been implemented as reactive Pluto~\cite{pluto} notebooks in Julia~\cite{bezanson2017julia} and are publicly available at \href{https://github.com/TUDelft-DeTAIL/AbTensors}{https://github.com/TUDelft-DeTAIL/AbTensors}. The notebook files can be freely downloaded and run on your local machine in Julia. An alternative way to use these notebooks that does not require the installation of Julia is to run them in the cloud via Binder~\cite{binder}. This can be done by clicking on each of the links on the main Github page. Please note that it can take over 10 minutes for Binder to download and compile all required packages.

As discussed in section~\ref{sec:parameterization} we parameterized the prior covariance matrix $\mat{P}_0$ with a single hyperparameter $\sigma_P$ in both Applications~\ref{exp:exp2} and \ref{ex:ex3}.
\subsection{Sampling structured tensor priors}
\label{exp:exp1}
In this first application we demonstrate how Theorems~\ref{theo:nullspace},~\ref{theo:covar1} and~\ref{theo:orthcovar} are used to sample the priors of different $(\mat{A},\mat{b})$-constrained tensors.
\begin{example}(\textbf{Lower triangular tensors})
A first example of an $(\mat{A},\mat{b})$-constrained tensor considered here are lower triangular tensors. From Definition~\ref{def:triangulartensor} we know that triangular tensors are described by 
\begin{align*}
\mat{A} = \begin{pmatrix}\mat{A}_1 \\ \mat{A}_2 \\ \vdots \\ \mat{A}_{D-1} \end{pmatrix} = \begin{pmatrix}
        \mat{S}  \otimes  \mat{I}_J  \otimes   \cdots  \otimes  \mat{I}_J\\
        \mat{I}_J \otimes  \mat{S} \otimes  \cdots \otimes  \mat{I}_J\\
        \vdots \\
        \mat{I}_J \otimes  \mat{I}_J \otimes  \cdots \otimes  \mat{S}
\end{pmatrix} \in \mathbb{R}^{\frac{(D-1)(J-1)J^{D-1}}{2} \times J^D}    
\end{align*}
and zero vector $\mat{b}$. The square root of the covariance matrix is built up by applying Algorithm~\ref{alg:multipleconstraints}, which considers only 1 block row of $\mat{A}$ at a time. The whole $\mat{A}$ matrix is therefore never explicitly made. In the notebook it is possible to sample lower triangular tensors with orders ranging from 2 up to 5 and dimensions 2 up to 6 by moving the corresponding sliders.

\end{example}
\begin{example}(\textbf{Tensors with known sum of entries})
In this example we sample tensors $\ten{W}$ for which the sum over the last index always adds up to a value of 1: 
\begin{align*}
  \forall j_1, j_2, \ldots, j_{D-1}: \sum_{j_D} w_{j_1, j_2, \ldots, j_D} = b_{j_1, j_2, \ldots, j_{D-1}}=1.  
\end{align*}
From Lemma~\ref{lem:constantsum} we know that in this case $\mat{A} = \mat{1}_J^T \otimes \mat{I}_J \otimes \cdots \otimes \mat{I}_J$. It is straightforward to verify that a basis for the right nullspace of $\mat{A}$ is
\begin{align*}
 \begin{pmatrix}  1 & 1 & \cdots & 1 \\ -1 & 0 & \cdots & 0 \\ 0 & -1 & \cdots &0 \\ 0 & 0 & \cdots & -1\end{pmatrix}\otimes I_J \otimes \cdots \otimes I_J.   
\end{align*}
Sampling the prior can now be done without every constructing a basis for the nullspace explicitly since
\begin{align*}
\sqrt{\mat{P}_0}\; \mat{x} &= \left( \begin{pmatrix}  1 & 1 & \cdots & 1 \\ -1 & 0 & \cdots & 0 \\ 0 & -1 & \cdots &0 \\ 0 & 0 & \cdots & -1\end{pmatrix}\otimes \mat{I}_J \otimes \cdots \otimes \mat{I}_J \right) \;\mat{x} \\ 
&=\begin{pmatrix} \mat{I}_{J^{D-1}} & \mat{I}_{J^{D-1}} & \cdots & \mat{I}_{J^{D-1}} \\ -\mat{I}_{J^{D-1}} & 0 & \cdots & 0 \\ 0 & -\mat{I}_{J^{D-1}} & \cdots &0 \\ 0 & 0 & \cdots & -\mat{I}_{J^{D-1}}\end{pmatrix} \; \begin{pmatrix} \mat{x}_1 \\ \mat{x}_2 \\ \vdots \\ \mat{x}_{J-1} \end{pmatrix}= \begin{pmatrix} \mat{x}_1 + \mat{x}_2 + \cdots + \mat{x}_{J-1} \\ -\mat{x}_1 \\ -\mat{x}_2 \\ \vdots \\ -\mat{x}_{J-1}\end{pmatrix}.    
\end{align*}
It is therefore sufficient to sample $\mat{x} \in \mathbb{R}^{(J-1)\,J^{D-1}}$ from a standard normal distribution and do the operations on the $J-1$ partitions of $\mat{x}$ as described above to generate the desired sample.
In the notebook one can change the order of the sampled tensor from 2 up to 5 and dimension from 5 up to 10 by using the corresponding sliders.
\end{example}
\begin{example}(\textbf{Symmetric tensors})
Symmetric tensors $\ten{W}$ are tensors for which entries are invariant under any index permutation. The permutation matrix $\mat{S}$ in the symmetric case consists of cyclic permutations where each each cycle contains the entry $w_{j_1, \ldots,j_D}$ and all entries with corresponding index permutations $w_{\pi (j_1,\ldots, j_D)}$. For example, in the case $D=2$ and $J=2$ the permutation matrix $\mat{S}$ consists of $3$ cyclic permutations
\begin{align*}
	w_{1,1} \mapsto w_{1,1}, \; w_{2,1} \mapsto w_{1,2}, \; w_{1,2} \mapsto w_{2,1}, \; 
	w_{2,2} \mapsto w_{2,2}.
\end{align*}
The order $K$ of $\mat{S}$ in this case is $2$ since $\mat{S}^2=I$. According to Theorem~\ref{theo:covar1} we then have that the square root of the covariance matrix is $\sqrt{\mat{P}_0} = \nicefrac {(\mat{S} + \mat{S}^2)}{2}.$ When $D=3$, the order $K$ of the corresponding permutation matrix is $6$ and hence $\sqrt{\mat{P}_0} = \nicefrac {(\mat{S} + \mat{S}^2 + \mat{S}^3 + \mat{S}^4 + \mat{S}^5 + \mat{S}^6)}{6}.$ Sampling from these priors is done via Algorithm~\ref{alg:fastsampling} where a standard normal sample $\mat{x} \in \mathbb{R}^{J^D}$ is generated and permuted $K$ times.
The notebook allows you to sample symmetric tensors of orders 2 and 3 and dimensions 3 up to 10.
\end{example}
\begin{example}(\textbf{Hankel tensors})
Hankel tensors $\ten{W}$ are tensors for which entries with a constant index sum $j_1+\cdots+j_D$ have the same numerical value. The order $K$ of the corresponding permutation matrix $\mat{P}$ grows very quickly. For example, when $D=2$ and $J=20$ the order $K$ is the least common multiple of $1,2,\ldots,20 = 232,792,560$. Theorem~\ref{theo:orthcovar}, however, allows us to construct a matrix $\sqrt{\mat{P}_0} \in \mathbb{R}^{J^D \times R}$, where $R$ is the number of permutation cycles. For Hankel tensors we have that $R = D(J-1)+1$.
The notebook allows you to sample Hankel tensors of order 2 up to 4 and dimensions 3 up to 10.
\end{example}
\subsection{Completion of a Hankel matrix from noisy measurements}
\label{exp:exp2}
Hankel matrices are very common in signal processing and control theory. In this application a Bayesian approach will be used to complete a Hankel matrix based on noisy incomplete measurements. For this we use the following forward model $\mat{y} = \mat{\Phi} \; \mat{w} + \mat{\epsilon}$, where $\mat{w} \in \mathbb{R}^{10^2}$ is the vectorization of the true underlying $10 \times 10$ Hankel matrix. The $ I \times 10^2$ matrix $\mat{\Phi}$ selects $I$ random entries of $\mat{w}$ with equal probability. Each row of $\mat{\Phi}$ contains a single nonzero unit-valued entry at a random location. The number of measurements $I$ can be changed through a slider in the notebook. The vector $\mat{\epsilon}$ is a vector of zero-mean Gaussian noise. Given $\mat{y}$ and $\mat{\Phi}$, a Bayesian estimate of the underlying Hankel matrix $\mat{W}$ can be obtained from~\eqref{eq:meanposterior} as the posterior mean $\mat{w}_+$. Another commonly used estimate is the maximum likelihood estimate, which is the $\mat{w}$ that maximizes the likelihood $p(\mat{y}|\mat{w})$. We compare two posterior estimates with the maximum likelihood estimate under two different assumptions on the noise covariance. We fix the sampling rate at $50\%$ and choose $\sigma_\epsilon^2=1$. The prior covariance matrix is set to $\sigma_P^2\,\mat{P}_0 = 10^{-6}\,\mat{P}_0$, where $\mat{P}_0$ is covariance matrix of the Hankel prior obtained via Theorem~\ref{theo:orthcovar}. 

\begin{example}(\textbf{White noise})
First we consider white noise, which implies that $\mat{\Sigma}=\sigma_\epsilon^2\,\mat{I}$.  The singular values of the prior precision $\nicefrac{\sqrt{\mat{P}_0^{-1}}}{\sigma_P}$, posterior precision $\begin{pmatrix}
        \nicefrac{\mat{\Phi}^T}{\sigma_\epsilon} & \nicefrac{\sqrt{\mat{P}_0^{-1}}^T}{\sigma_P}
    \end{pmatrix}^T$, and likelihood precision $\nicefrac{\mat{\Phi}}{\sigma_\epsilon}$
are shown in Figure~\ref{fig:svals_precsion}. They provide us with insight on how the prior, posterior and likelihood relate to each other. The likelihood $p(\mat{y}|\mat{w})$ only has 50 measurements and gives all of them equal weight. The prior $p(\mat{w})$ on the other hand only considers 19 nonzero values as a $10 \times 10$ Hankel matrix has 19 distinct entries. Given the relative high noise variance compared to the prior, the posterior $p(\mat{w}|\mat{y})$ ``follows" the prior for the first 19 singular values.
\begin{figure}
     \centering
     \begin{subfigure}[b]{0.49\textwidth}
         \centering
         \includegraphics[width=\textwidth]{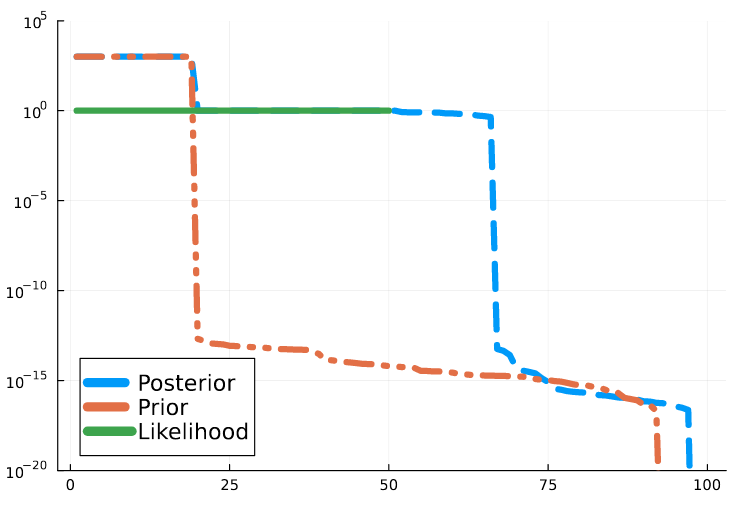}
         \caption{White-noise case. Given the relative high noise variance the posterior follows the prior for the first 19 singular values.}
         \label{fig:svals_precsion}
     \end{subfigure}
     \hfill
     \begin{subfigure}[b]{0.49\textwidth}
         \centering
     \includegraphics[width=\textwidth]{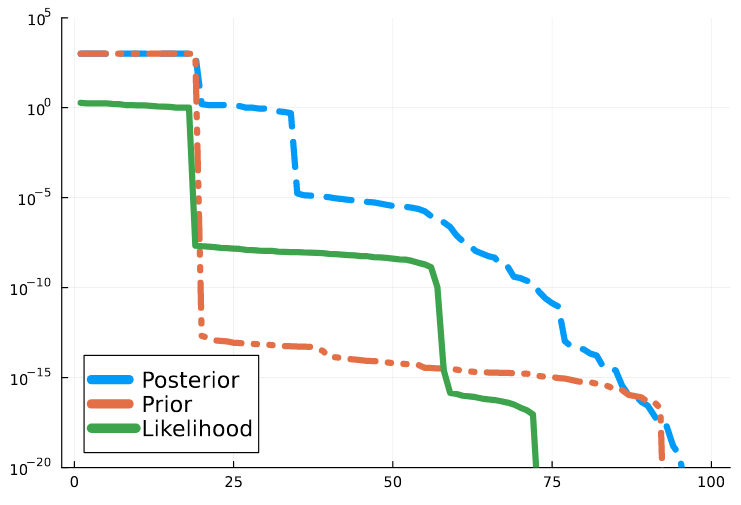}
         \caption{Hankel-noise case. Also in this case we have that the posterior follows the prior for the first 19 singular values.}
         \label{fig:svals_precsion2}
     \end{subfigure}
        \caption{Singular values of the square-root precision matrices of the prior, likelihood and posterior distribution. Only $50\%$ of the Hankel matrix $\mat{W}$ was measured. The noise variance is 1 and the prior variance is $10^{-6}$.}
    \label{fig:svals_exp2}
\end{figure}
A prior mean is obtained by averaging over the nonzero antidiagonals of the measurements and using those averages to construct a Hankel matrix. We now compute three different estimates and compare them to the ground truth. The first estimate is obtained from~\eqref{eq:meanposterior} with a backslash solve. A second estimate is computed by truncating the SVD of $\begin{pmatrix}\nicefrac{\Phi^T}{\sigma_\epsilon} &
        \nicefrac{{P}_0^{-T}}{\sigma_P}
\end{pmatrix}^T$ to rank 19 in equation~\eqref{eq:meanposterior}. The third estimate is the maximum likelihood estimate. For each of these estimates we show the relative error in Table~\ref{tab:hankel}.
\begin{table}[t]
\centering
\caption{Relative errors for three different Hankel matrix completion estimates $\mat{\hat{w}}$. Smallest relative error is indicated in bold.}
\label{tab:hankel} 
\begin{tabular}{@{}lrrr@{}}
                    & backslash  & truncated SVD    & max-likelihood \\\midrule 
$\frac{||\mat{w}-\mat{\hat{w}}||_2}{||\mat{w}||_2}$ (white noise)    & $ 0.160$     & $\textbf{0.137}$ & $0.614$     \\[2ex]
$\frac{||\mat{w}-\mat{\hat{w}}||_2}{||\mat{w}||_2}$ (Hankel noise)    & $ 0.235$     & $\textbf{0.137}$ & $0.604$     \\[2ex]
$\frac{||\mat{H}\mat{\hat{w}}-\mat{\hat{w}}||_2}{||\mat{\hat{w}}||_2}$         & $0.12$     & $6.3\text{e-}7$ & $0.80$     
\end{tabular}
\end{table}   
Adding the Hankel prior shows a clear improvement on the completed Hankel matrix. The relative error is 4 times smaller from the inclusion of the prior. Since the first 19 singular values of the posterior are equal to the singular values of the prior one could expect the estimated posterior mean $\mat{w}_+$ obtained from truncating the SVD to the first 19 singular values to be Hankel. In order to confirm this we also compute the relative Hankel error $\nicefrac{||\mat{H}\,\mat{w}-\mat{w}||_2}{||\mat{w}||_2}$ for the three estimates in Table~\ref{tab:hankel}, where $\mat{H}$ is the Hankel permutation matrix. Restricting the posterior mean to lie in a subspace spanned by the first 19 right singular vectors indeed enforces a Hankel structure.
\end{example}

\begin{example}(\textbf{Hankel distributed noise})
To investigate the effect of the noise covariance on the estimates we now consider noise $\mat{e}$ that also has a Hankel structure. In other words, the covariance matrix for $p(\mat{e})$ is $\sigma_\epsilon^2\, \mat{P}_0$, whereas the prior covariance is $\sigma_P^2\, \mat{P}_0$.
With the noise being Hankel, this means that the perturbation $\mat{\epsilon}$ of $\mat{w}$ will have a Hankel structure as well. This can be modeled via the forward model $\mat{y} = \Phi (\mat{w} + \mat{\epsilon})$, where now $p(\Phi \mat{\epsilon}) = \mathcal{N}(0, \sigma_\epsilon^2\, \Phi\, P_0\, \Phi^T)$. Figure~\ref{fig:svals_precsion2} shows the singular values of the square-root precision matrices. The number of nonzero singular values of the likelihood now consists of 2 plateaus. Again, the posterior follows the prior for the first 19 singular values. Since now measurements of entries along the same antidiagonal are identical, less information is to be extracted from the measurements. This explains the first drop of Figure~\ref{fig:svals_precsion2} at the 19th singular value for both the likelihood and posterior. Less information also means that we can expect our estimate to be worse compared to the white noise case. The relative errors are now indeed higher, as seen in Table~\ref{tab:hankel}.
Note however that the estimate obtained by truncating the SVD remains the same.
\end{example}
\subsection{Bayesian learning of MNIST classifier}
\label{ex:ex3}
In this application we learn a classifier for images of $10$ handwritten digits. The classifier is trained on the MNIST data~\cite{lecun1998gradient}, which consists of $60,000$ pictures for training and $10,000$ pictures for test. Each picture $\mat{x}_n$ is of size $28 \times 28$. We pick $10,000$ random samples from the training set and convert each picture $\mat{x}_n$ into $25^2=625$ Random Fourier Features $\mat{\varphi}(\mat{x}_n)_j = \text{Re}(e^{-i\,\mat{v}_j^T\mat{x}_n})$~\cite{rahimi2007random}. The 625 frequency vectors $\mat{v}_j$ are sampled from a zero-mean Gaussian with variance $\nicefrac{1}{5^2}\,\mat{I}$. We use a one-vs-all strategy by learning $10$ classifiers at once. Each classifier is trained to distinguish between $1$ particular class versus all others. The forward model for our $10$ classifiers is then $\mat{y} = \mat{\varphi}(\mat{x}) \; \mat{W} + \mat{e}$. Each column of $\mat{W} \in \mathbb{R}^{625 \times 10}$ contains the model parameters of $1$ specific classifier. In order to predict the class of a sample $\mat{x}^*$ we compute $\mat{y}^* = \mat{\varphi}(\mat{x}^*)\,\mat{W}$ and apply the softmax function
\begin{align*}
	\mat{\sigma}(\mat{y}^*) = \frac{e^{\mat{y}^*_k}} {\sum_k e^{\mat{y}^*_k}} \in \mathbb{R}^{10}.
	\end{align*}
The prediction is then the class with maximal $\mat{\sigma}(\mat{y}^*)$. The 10 classifiers are trained on a training data set of pictures $\mat{X} \in \mathbb{R}^{10,00 \times 784}$ and corresponding class labels $\mat{Y} \in \mathbb{R}^{10,000 \times 10}$. Our estimate for $\mat{W}$ is the mean of the posterior $p(\mat{W} | \mat{Y}, \mat{X} )$. The residual $\mat{e}$ is most commonly assumed to be zero-mean white Gaussian noise $p(\mat{e})=\mathcal{N}(0,\sigma_\epsilon^2\mat{I})$. Likewise, the prior $p(\mat{W})$ is usually assumed to be a zero-mean normal distribution with a uniform scaling covariance matrix $\mat{P}_0 = \sigma_P^2\;\mat{I}$. Such a prior is also called Tikhonov regularization.
We compare the performance of the Tikhonov prior to other zero-mean $(\mat{A},\mat{b})$-constrained tensor priors (symmetric, Hankel en circulant), constructed using either Theorem 4.5 or Theorem 5.1. The noise variance $\sigma_\epsilon^2$ is set to a fixed value of 1. 
\begin{figure}
     \centering
     \begin{subfigure}[b]{0.49\textwidth}
         \centering
         \includegraphics[width=\textwidth]{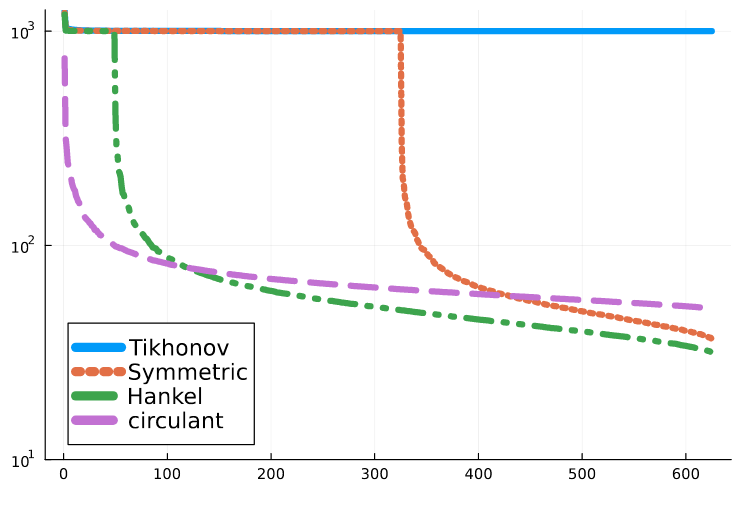}
         \caption{When $\sigma_P^2=10^{-6}$ large differences between the different posteriors are observed. The corresponding classifiers are therefore expected to also behave differently.}
         \label{fig:svals_precsion3}
     \end{subfigure}
     \hfill
     \begin{subfigure}[b]{0.49\textwidth}
         \centering
     \includegraphics[width=\textwidth]{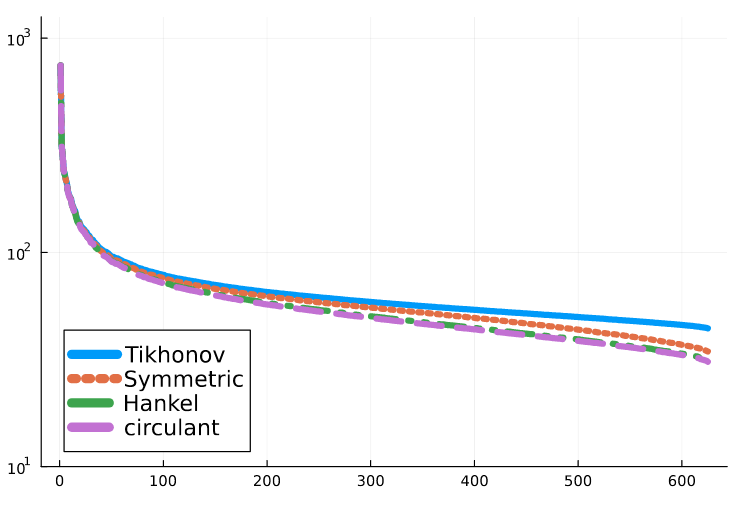}
         \caption{When $\sigma_P^2=10^{-3}$ all differences between the different posteriors have almost vanished. The corresponding classifiers are expected to also behave similarly.}
         \label{fig:svals_precsion4}
     \end{subfigure}
        \caption{Singular values of the square-root precision matrices of the posterior distribution for 4 different priors. The noise variance is fixed to 1.}
    \label{fig:svals_exp3}
\end{figure}
The difference between these different priors can be investigated by looking at the singular value profiles of the square-root precision matrices of the corresponding posteriors. These are shown in Figure~\ref{fig:svals_precsion3} for $\sigma_P^2=10^{-6}$ and in Figure~\ref{fig:svals_precsion4} for $\sigma_P^2=10^{-3}$. Being confident in the prior ($\sigma_P^2=10^{-6}$) has a strong effect on the corresponding posterior, which explains the large differences in singular value profiles. The corresponding classifiers can then be expected to also differ a lot on unseen test data. Indeed, applying the obtained classifiers on $10,000$ test images results in a relative number of correctly classified images shown in Table~\ref{tab:mnist}.
\begin{table}[t]
\centering
\caption{Comparison of relative number of correctly classified images for classifiers learned with different priors. Best classifier indicated in bold.}
\label{tab:mnist} 
\begin{tabular}{@{}lrrrr@{}}
                    & Tikhonov  & symmetric    & Hankel & circulant \\\midrule 
$\sigma_P^2 = 10^{-6}$    & $ 0.650$     & $0.880$ & $\textbf{0.917}$ &  $0.915$    \\[2ex]
$\sigma_P^2 = 10^{-3}$    & $ 0.917$     & $0.918$ & $\textbf{0.920}$ & $0.919$ 
\end{tabular}
\end{table}   
All $(\mat{A},\mat{b})$-constrained priors outperform the conventional Tikhonov prior, with Hankel and circulant tensors having the best performance. By increasing the prior covariance to $\sigma_P^2=10^{-3}$ all singular value profiles become very similar. The corresponding classifiers have similar performance as seen in Table~\ref{tab:mnist}.
No significant classification improvement is observed for the Hankel and circulant priors.
\section{Conclusions}
A whole new class of Bayesian priors has been worked-out which could be potentially applied to a variety of different inverse problems. The main focus of this article was mostly on the theoretical foundation and where possible we discussed practical implementations without going into much detail. Although the curse of dimensionality when considering tensors of large order and dimension can be completely resolved via the corresponding dual problem, the computational complexity can still become prohibitively large with increasing sample size. To tackle this complexity the possibility to represent the prior mean vector and covariance matrix of these priors as exact low-rank tensor decompositions could be investigated. 

\section*{Acknowledgments}
Many thanks to Frederiek Wesel for valuable discussions and feedback.

\bibliographystyle{siamplain}
\bibliography{references}

\end{document}